\documentclass[11pt,a4paper]{article}
\usepackage[cp1251]{inputenc}
\usepackage{amsmath,amsthm}
\usepackage{amsfonts,amstext,amssymb,verbatim,epsfig}
\usepackage{dsfont}
\def\Z{{\mathbb{Z}}}
\newcommand{\A}{\mathrm{Ann}}
\newcommand{\ipc}{\mathrm{IPC}}
\newcommand{\iic}{\mathrm{IIC}}
\newcommand{\event}{\mathcal{E}}
\newcommand{\vset}{\mathrm{S}}

\newcommand{\ball}{\mathrm{B}}
\newcommand{\pn}[1]{p_{\scriptscriptstyle #1}}

\newtheorem{theorem}{Theorem}

\newtheoremstyle{likedef}
  {}%
  {}%
  {}%
  {\parindent}%
  {\bfseries}%
  {.}%
  {.5em}%
  {}%

\theoremstyle{likedef}

\newtheorem{remark}{Remark}

\begin{document}
\title{The incipient infinite cluster does not stochastically dominate the invasion percolation cluster in two dimensions}

\author{Art\"{e}m Sapozhnikov\thanks{ETH Z\"urich, Department of Mathematics, R\"amistrasse 101, 8092 Z\"urich. Email: artem.sapozhnikov@math.ethz.ch.
The research of the author has been supported by the grant ERC-2009-AdG  245728-RWPERCRI.}
}

\maketitle

\footnotetext{MSC2000: Primary 82C43, 60K35, 82B27, 82B43.}
\footnotetext{Keywords: Invasion percolation, incipient infinite cluster, critical percolation, near-critical percolation, 
correlation length, stochastic domination.}

\begin{abstract}
This note is motivated by results in \cite{tree,DSV} about 
global relations between the invasion percolation cluster (IPC) and 
the incipient infinite cluster (IIC) on regular trees and on two dimensional lattices, respectively.
Namely, that the laws of the two objects are mutually singular, and, in the case of regular trees, 
that the IIC stochastically dominates the IPC. 
We prove that on two dimensional lattices, the IIC does not stochastically dominate the IPC. 
This is the first example showing that the relation between the IIC and IPC is 
significantly different on trees and in two dimensions. 
\end{abstract}

\bigskip
\bigskip
\bigskip

In the classical mathematical theory of percolation, the edges (or vertices) of an infinite lattice are deleted independently with 
probability $1-p$, and the properties of the remaining components are studied. 
There is a phase transition in the parameter: if $p$ is bigger than some critical value $p_c$, 
there is an infinite component with probability $1$ (in this case we say that percolation occurs), and 
if $p$ is below $p_c$, the probability of the existence of an infinite component is $0$. 
If $p=p_c$, the geometric properties of connected components are highly non-trivial. 
It is expected (and has been proved for certain lattices, see, e.g., \cite{LSW:onearm,SW}) 
that simple characteristics of connected components 
(e.g., diameter, volume) obey power laws. In other words, at criticality, connected components are self-similar random objects. 

\medskip

Invasion percolation is a stochastic growth model \cite{Chandler,Lenormand} 
that mirrors aspects of the critical percolation picture without tuning any parameter. 
To define the model, let $G=(V,E)$ be an infinite connected graph in which a distinguished vertex, the origin, is chosen. 
The edges of $G$ are assigned independent uniform random variables $\tau_e$ on $[0,1]$, called weights. 
(The underlying probability measure is denoted by $\mathbb P$.) 
The {\it invasion percolation cluster} (IPC) of the origin on $G$ is defined as the limit of an increasing sequence $(G_n)$ of 
connected sub-graphs of $G$ as follows. 
Define $G_0$ to be the origin. 
Given $G_n = (V_n,E_n)$, the edge set $E_{n+1}$ is obtained from $E_n$ by adding to it 
the edge from the set $\{(x,y)\in E\setminus E_n~:~x\in V_n, y\in V\}$ with the smallest weight. 
Let $G_{n+1}$ be the graph induced by the edge set $E_{n+1}$.

Invasion percolation is closely related to independent bond percolation.
For any $p\in[0,1]$ we say that an edge $e\in E$ is $p$-open if $\tau_e< p$ or $p$-closed if $\tau_e \geq p$.
The resulting random graph $(V, \{e\in E~:~\tau_e < p\})$ of $p$-open edges has 
the same distribution as the one obtained from $G$ by deleting edges independently with probability $1-p$.
It was shown in \cite{Newman} that for all $p>p_c$ the IPC on $\Z^d$ intersects the infinite $p$-open cluster with probability $1$.
This result was extended to a much more general class of graphs in \cite{HPS}.
It is also easy to see that once a vertex in an infinite $p$-open cluster is invaded,
all further invaded edges are in the infinite $p$-open cluster that contains this vertex.
The above observations imply that $\limsup_{n\to\infty}\tau_{e_n}=p_c$, where $e_n$ is an edge invaded at step $n$.
This indicates 
that invasion dynamics reproduce the critical percolation picture, 
a phenomenon often referred to as {\it self-organized criticality}. 
A comprehensive analysis of invasion percolation and its relation to critical independent percolation 
were obtained in \cite{BJV,DS,DS:CLT,DSV,Jarai,Zhang} for two dimensional lattices, and in \cite{tree,Goodman} for regular trees. 

\medskip

Results of \cite{Newman} also indicate that a large proportion of the edges in the IPC belongs to big $p_c$-open clusters.
Very large $p_c$-open clusters are usually referred to as {\it incipient infinite clusters} (IIC).
The first mathematical definition of the IIC of the origin in two dimensions was given by Kesten in \cite{KestenIIC}:
(1) by conditioning on the $p_c$-open cluster of the origin being connected to a vertex at distance $n$ from the origin and letting $n\to\infty$, or
(2) by conditioning on the $p$-open cluster of the origin being infinite and letting $p\downarrow p_c$.
These are two different constructions of the same measure on configurations of open and closed edges.
Under this measure the open cluster of the origin is almost surely infinite.

Relations between the IIC and the IPC in two dimensions were first observed in \cite{Jarai,Zhang}.
The scaling of the moments of the number of invaded sites in a box was obtained there, which turned out to be
the same as the scaling of the corresponding moments for the IIC.
Similar analysis shows that the $k$-point functions of the IIC and the IPC are comparable.
Corresponding results about local similarities of the IIC and IPC on a regular tree are obtained in \cite{tree}.

\bigskip

The result of this note is motivated by global relations between the IIC and the IPC on regular trees and in two dimensions, 
discovered respectively in \cite{tree} and \cite{DSV}. 
In \cite{tree}, it was shown that on regular trees, 
the IIC stochastically dominates the IPC, and the laws of the IIC and IPC are mutually singular. 
In \cite{DSV}, it was proved that in two dimensions, the laws of the IIC and the IPC are also mutually singular. 
The proof in \cite{DSV} also implies that the IPC does not stochastically dominate the IIC. 
In fact, the proof in \cite{DSV} strongly suggests that in two dimensions one should expect the same behavior as 
on regular trees, namely, that the IIC should stochastically dominate the IPC. 
The main result of this note is a disproof of the above belief. 
In Theorem~\ref{thm:IPCandIIC} we prove that {\it the IIC does not stochastically dominate the IPC in two dimensions}. 
The key idea behind the proof is that finite connected components of large volume are more likely to be seen 
in supercritical percolation than in the critical one. 

\bigskip

In the remainder, we will formally define independent percolation, 
recall the definition of Kesten's IIC from \cite{KestenIIC}, 
state and prove the main result of this note. 
For simplicity {\it we restrict ourselves here to the square lattice}.
The result of this note still holds for lattices which are invariant under reflection in one of the coordinate axes
and under rotation around the origin by some angle in $(0,\pi)$. 
In particular, this includes the triangular and honeycomb lattices.

\medskip

We consider the square lattice $(\Z^2,{\mathbb E}^2)$, where ${\mathbb E}^2 = \{(x,y)\in\Z^2\times\Z^2~:~|x-y|=1\}$, 
which we simply denote by $\Z^2$. 
For $p\in[0,1]$, we consider a probability space $(\Omega_p,{\mathcal F}_p,{\mathbb P}_p)$, where
$\Omega_p = \{0,1\}^{{\mathbb E}^2}$, ${\mathcal F}_p$ is the $\sigma$-field generated by the finite-dimensional
cylinders of $\Omega_p$, and ${\mathbb P}_p$ is a product measure on $(\Omega_p,{\mathcal F}_p)$,
${\mathbb P}_p = \prod_{e\in{\mathbb E}^2}\mu_e$, where $\mu_e$ is given by $\mu_e(\omega_e = 1) = 1 - \mu_e(\omega_e = 0) = p$, for
vectors $(\omega_e)_{e\in{\mathbb E}^2}\in\Omega_p$.
We say that an edge $e$ is {\it open} if $\omega_e = 1$, and $e$ is {\it closed} if $\omega_e = 0$.
The event that two sets of sites $\vset_1,\vset_2\subset\Z^2$ are connected by an open path
is denoted by $\vset_1 \leftrightarrow \vset_2$, and 
the complement of this event is denoted by $\vset_1 \nleftrightarrow \vset_2$.
If a set $\vset$ is a singleton $\{x\}$, then we simply write $\vset = x$. 
In particular, we write $x\leftrightarrow y$ for the event that $x$ and $y$ are connected by an open path. 
The percolation probability $\theta(p)$ is the probability under $\mathbb P_p$ that the open cluster of the origin is infinite. 
There exists a critical probability $p_c\in (0,1)$ such that $\theta(p)>0$ if and only if $p>p_c$. 
(In particular, $\theta(p_c) = 0$ in two dimensions, as proved by Kesten in \cite{Kesten1/2}. 
This result has also been proved for percolation on $\Z^d$, with $d\geq 19$, 
by Barsky and Aizenman \cite{BarAiz91} and Hara and Slade \cite{HS}.
The question whether $\theta(p_c)=0$ for $d\in[3,18]$ still remains open.)
We refer the reader to \cite{Grimmett} for background on percolation.

\medskip

Let $\ball(n) = [-n,n]^2$ be the $l^\infty$-ball of radius $n$ centered at the origin and $\partial \ball(n) = \ball(n)\setminus \ball(n-1)$ 
its internal boundary. 
It is shown by Kesten in \cite{KestenIIC} that the limit
\begin{equation*}
\nu\left[\event\right] = \lim_{N\to\infty}{\mathbb P}_{p_c}\left[\event~|~0\leftrightarrow \partial \ball(N)\right]
\end{equation*}
exists for any event $\event$ that depends on the state of finitely many edges in ${\mathbb E}^2$.
The unique extension of $\nu$ to a probability measure on configurations of open and closed edges
exists. Under this measure, the open cluster of the origin is almost surely infinite.
It is called the {\it incipient infinite cluster} (IIC).

\bigskip

The main result of this note is the following theorem. 
\begin{theorem}\label{thm:IPCandIIC}
The IPC is not stochastically dominated by the IIC on $\Z^2$, i.e., 
there is no coupling such that the IIC contains the IPC with probability $1$.  
\end{theorem}

\bigskip

\begin{remark}
In \cite{tree}, it was proved that on a regular tree, 
the IIC stochastically dominates the IPC, 
and the laws of the IIC and the IPC are mutually singular. 
The analysis in \cite{tree} is based on 
a representation of the IPC as an infinite backbone rooted at the origin 
with finite branches attached to it, which are distributed as 
subcritical percolation clusters 
with varying percolation parameter described by the so-called forward maximal weight process along the backbone.
Since the IIC on a regular tree is described as an infinite backbone from the origin 
with critical percolation clusters attached to it, 
the stochastic domination immediately follows from these representations.  
The proof of mutual singularity is based on a deep analysis of the distribution of the finite branches of the IPC. 
\end{remark}
\begin{remark}
In \cite{DSV}, it is shown that the laws of the IPC and the IIC are mutually singular. 
The main ingredient of the proof of this result is \cite[Theorem~7]{DSV}, which states the following. 
Let $G = (V,E)$ be an infinite connected subgraph of $(\Z^2,{\mathbb E}^2)$ which contains the origin.
We call an edge $e\in E$ a {\it disconnecting edge} for $G$
if the graph $(V,E\setminus\{e\})$ has a finite component, and if the origin belongs to this finite component.
Let $\mathcal{D}_\iic(m,n)$ be the event that the IIC does not contain a disconnecting edge in the annulus $\ball(n)\setminus \ball(m)$, and let
$\mathcal{D}_\ipc(m,n)$ be the event that the IPC does not contain a disconnecting edge in the annulus $\ball(n)\setminus \ball(m)$.
Then \cite[Theorem~7]{DSV} states that there exists a sequence $(n_k)$ such that
\[
{\mathbb P}\left[\sum_k {\mathds 1}(\mathcal{D}_\ipc(n_k, n_{k+1})) < \infty\right] = 1 
\quad\mbox{and}\quad
\nu\left[\sum_k {\mathds 1}(\mathcal{D}_\iic(n_k, n_{k+1})) = \infty\right] = 1 .\
\]
The above statement implies that the IIC is supported on clusters for which
infinitely many of the events $\mathcal{D}_\iic(n_k, n_{k+1})$ occur, and 
the IPC is supported on clusters for which only finitely many
of the events $\mathcal{D}_\ipc(n_k, n_{k+1})$ occur. 
In particular, this immediately implies that 
there is no coupling such that the IPC contains the IIC with probability $1$.
\end{remark}

\bigskip

\begin{proof}[Proof of Theorem~\ref{thm:IPCandIIC}]
For integers $m<n$, let $\A(m,n) = \ball(n)\setminus \ball(m)$.
Let $\event_\ipc(n)$ be the event that $\A(n,2n)$ is a subgraph of the IPC (i.e., all the edges of $\A(n,2n)$ are in the IPC), 
and let $\event_\iic(n)$ be the event that $\A(n,2n)$ is a subgraph of the IIC. 
We will prove that there exists $n$ such that 
\begin{equation}\label{eq:IPCandIIC:relation}
\mathbb P[\event_\ipc(n)]\geq 2~\nu[\event_\iic(n)] > 0 ,\
\end{equation}
which will complete the proof of Theorem~\ref{thm:IPCandIIC}. 

\medskip

For $p>p_c$ and $\varepsilon>0$, let 
\[
L(p) = \min\{n~:~ \mathbb P_p[\text{there exists an open left-right crossing of $[0,n]^2$}]\geq 1-\varepsilon\} 
\]
be the finite-size scaling correlation length (see, e.g., \cite[(1.21)]{kesten}), and let $\pn{n}=\sup\{p~:~ L(p)>n\}.$
If $\varepsilon$ is small enough, $L(p)\to\infty$ as $p\to p_c$ and $L(p)=1$ for $p$ close to $1$. 
In particular, the function $p_n$ is well-defined. 
From now on, we fix such $\varepsilon$. By the definition of $L(p)$ and $p_n$, for all $n$ and $m\leq n$, we have 
\[
\mathbb P_{\pn{n}}[\text{there exists an open left-right crossing of $[0,m]^2$}]\leq 1-\varepsilon .\
\]
Therefore, it follows from the planar duality (see, e.g., \cite[Section~11.2]{Grimmett}) and 
the RSW theorem (see, e.g., \cite[Section~11.7]{Grimmett}) that there exists $c>0$ such that 
\begin{equation}\label{eq:RSWforpn}
\mathbb P_{\pn{n}} \left[\ball(2n)\nleftrightarrow \partial \ball(4n)\right] \geq c > 0 ,\quad\mbox{for all }n\geq 1 .\
\end{equation}
Let $\event(n)$ be the event that all the edges of $\A(n,2n)$ are open. 
Note that for any $p\in(0,1)$, by the definition of invasion, 
the event $\event_\ipc(n)$ occurs if (a) all the edges of $\A(n,2n)$ are $p$-open, and 
(b) $\ball(2n)$ is not connected to $\partial \ball(4n)$ by a $p$-open path. 
In particular, by taking $p = \pn{n}$, independence and \eqref{eq:RSWforpn} give 
\begin{equation}\label{eq:IPC:lowerbound}
\mathbb P[\event_\ipc(n)] \geq 
c ~\mathbb P_{\pn{n}}[\event(n)]  = c ~\pn{n}^{|\A(n,2n)|} .\
\end{equation}
On the other hand, by the definition of the IIC measure and independence, we have 
\begin{equation}\label{eq:IIC:relation}
\nu[\event_\iic(n)] = \lim_{N\to\infty} 
\frac{\mathbb P_{p_c}[\event(n), 0\leftrightarrow \partial \ball(N)]}{\mathbb P_{p_c}[0\leftrightarrow \partial \ball(N)]}
=
\mathbb P_{p_c}[\event(n)] \lim_{N\to\infty}
\frac{\mathbb P_{p_c}[0\leftrightarrow \partial \ball(n)]\mathbb P_{p_c}[\ball(2n)\leftrightarrow \partial \ball(N)]}
{\mathbb P_{p_c}[0\leftrightarrow \partial \ball(N)]} .\
\end{equation}
By the RSW theorem and a standard gluing argument (see \cite[Section~11.7]{Grimmett} or \cite[(29)]{KestenIIC}), 
there exists $C<\infty$ such that for all $n\geq 1$ and $N\geq 2n$, 
\[
\mathbb P_{p_c}\left[0\leftrightarrow \partial \ball(N)\right]
\leq
\mathbb P_{p_c}\left[0\leftrightarrow \partial \ball(n)\right]
~\mathbb P_{p_c}\left[\ball(2n)\leftrightarrow \partial \ball(N)\right]
\leq 
C~\mathbb P_{p_c}\left[0\leftrightarrow \partial \ball(N)\right] .\
\]
This property is usually referred to as quasi-multiplicativity of the one-arm probability. 
Plugging in the above relations into \eqref{eq:IIC:relation} gives
\begin{equation}\label{eq:IIC:bounds}
\mathbb P_{p_c}[\event(n)]\leq
\nu[\event_\iic(n)]
\leq
C~\mathbb P_{p_c}[\event(n)] 
= C~ p_c^{|\A(n,2n)|} .\
\end{equation}
In particular, for any $n$, $\nu[\event_\iic(n)]>0$. 
We conclude from \eqref{eq:IPC:lowerbound} and \eqref{eq:IIC:bounds} that  
\[
\mathbb P[\event_\ipc(n)] \geq c'~ (\pn{n}/p_c)^{|\A(n,2n)|}~ \nu[\event_\iic(n)]
\geq c'~ e^{c'(\pn{n}-p_c)n^2}~ \nu[\event_\iic(n)] ,\
\]
for some constant $c'>0$. 
It follows from \cite[(4.5)]{kesten} (see also \cite[Proposition~3.2]{Nolin}) and the fact that $\theta(p_c)=0$ 
that 
\[
(\pn{n}-p_c)n^2 \to \infty\quad \mbox{as }n\to\infty .\
\] 
In particular, there exists $n$ such that \eqref{eq:IPCandIIC:relation} holds. 
This completes the proof of Theorem~\ref{thm:IPCandIIC}. 
\end{proof}

\medskip

\begin{remark}
Instead of showing \eqref{eq:IPCandIIC:relation}, 
it would be more natural to prove that there exists $n$ such that 
$\mathbb P[\widetilde \event_\ipc(n)]\geq 2~\nu[\widetilde \event_\iic(n)] > 0$, where 
$\widetilde \event_\ipc(n)$ is the event that $\ball(n)$ is a subgraph of the IPC, and $\widetilde \event_\iic(n)$ is 
the event that $\ball(n)$ is a subgraph of the IIC. 
This can be proved similarly to \eqref{eq:IPCandIIC:relation}. 
However, since the proof of \eqref{eq:IPCandIIC:relation} is slightly simpler, 
we decided to consider events $\event_\ipc(n)$ and $\event_\iic(n)$ 
instead of $\widetilde \event_\ipc(n)$ and $\widetilde \event_\iic(n)$. 
\end{remark}

\paragraph{Acknowledgements.}
I thank Michael Damron for enlightening discussions and comments on the draft.

\end{document}